\documentclass[letterpaper, 10 pt, conference]{ieeeconf}  

\IEEEoverridecommandlockouts                              
\usepackage{amsmath} 
\usepackage{amssymb}  
\usepackage{xcolor}

\usepackage{algorithm} 
\usepackage{algpseudocode} 

\usepackage{booktabs}
\usepackage{array}

\newtheorem{proposition}{Proposition}
\newtheorem{corollary}{Corollary}

\title{\LARGE \bf
Accurate and Warm-Startable Linear Cutting-Plane Relaxations for ACOPF
}

\author{Daniel Bienstock and Mat\'ias Villagra
\thanks{This work was supported by an ARPA-E GO Competition Grant.}
 }

\begin{document}

\maketitle
\thispagestyle{empty}
\pagestyle{empty}

\begin{abstract}

We present a linear cutting-plane relaxation approach that rapidly proves tight lower bounds for the Alternating Current Optimal Power Flow Problem (ACOPF). Our method leverages outer-envelope linear cuts for  well-known second-order cone relaxations for ACOPF along with modern cut management techniques. These techniques prove effective on a broad family of ACOPF instances, including the largest ones publicly available, quickly and robustly yielding sharp bounds. Our primary focus concerns the (frequent) case where an ACOPF instance is considered following a small or moderate change in problem data, e.g., load changes and generator or branch shut-offs.  We provide significant computational evidence that the cuts computed on the prior instance provide an effective warm-start for our algorithm.


\end{abstract}

\section{INTRODUCTION}

The Alternating-Current Optimal Power Flow (ACOPF) problem~\cite{c1} is a well-known challenging computational task. It is nonlinear, non-convex and with feasible region that may be disconnected; see~\cite{c2},~\cite{c3}. From a theoretical perspective, in~\cite{c4, c5} it is shown that the feasibility problem is strongly NP-hard;  ~\cite{c6} proved that it is weakly NP-hard on star-networks. In the current state-of-the-art, some interior point methods are empirically successful at computing excellent solutions but cannot provide any bounds on solution quality.   

At the same time, strong lower bounds are available through second-order cone (SOC) relaxations~\cite{c7, c8}; however all solvers do struggle when handling such relaxations for large or even medium cases (see~\cite{c9}; we will provide additional evidence for this point). Other techniques, such as spatial-branch-and-bound methods applied to McCormick (linear) relaxations of quadratically-constrained formulations for ACOPF, tend to yield poor performance unless augmented by said SOC inequalities \textit{and} interior point methods, the latter for upper bounds.

In this paper we present a fast (linear) cutting-plane method used to obtain tight relaxations for even the largest ACOPF instances, by appropriately approximating the SOC relaxations. The emphasis on linearly constrained formulations is motivated by the fact that, whereas the tight SOC relaxations for ACOPF are clearly challenging, linear programming technology is, at this point, very mature -- many LP solvers are able to handle massively large instances quickly and robustly; these attributes extend to the case where formulations are dynamically constructed and updated, as would be the case with a cutting-plane algorithm.  As we will show herein, our approach is both fast and accurate. 
 
 Moreover, the central focus on this paper concerns \textit{reoptimization}. In power engineering practice it is often the case that a power flow problem is solved on data that reflects a recent, and likely limited, update on a case that was previously handled. In short, the current problem instance is not addressed 'from scratch.' Our algorithm can naturally operate in \textit{warm-started} mode, i.e., make use of previously computed cuts to obtain sharp bounds more rapidly than from scratch.  

As an additional attribute arising from our work the fact that our formulations are linear paves the way for effective \textit{pricing} schemes, i.e., extensions of the LBMP pricing setup currently used in energy markets~\cite{c10,c11,c12}.

\subsection{Our contributions}

\begin{itemize}
\item We describe very tight linearly constrained relaxations for ACOPF.  The relaxations can be constructed and solved robustly and quickly via a cutting-plane algorithm that relies on proper cut management. On medium to (very) large instances our algorithm is competitive or better, from scratch, with what was previously possible using nonlinear relaxations, both in terms of bound quality and solution speed.
\item We provide a theoretical justification for the tightness of the SOC relaxation for ACOPF as well as for the use of our linear relaxations.  
\item As a main contribution we demonstrate, through extensive numerical testing, that the warm-start feature for our cutting-plane algorithm yields tight bounds far faster than otherwise possible. It is worth noting that this capability stands in contrast to what is possible using nonlinear (convex) solvers (cf. Tables~\ref{table:ws_loads_1_1} and~\ref{table:ws_onelineoff}).
\end{itemize}

 
\section{ACOPF Problem Formulation and Relaxations}


\subsection{ACOPF}

We denote by $\mathcal{N} := (\mathcal{B},\mathcal{E})$ the network, where $\mathcal{B}$ denotes the set of buses, and $\mathcal{E}$ denotes the set of branches. We denote by $\mathcal{G}$ the set of generators of the grid, each of which is located at some bus;  for each bus $k \in \mathcal{B}$, we denote by $\mathcal{G}_{k} \subseteq \mathcal{G}$ the generators at bus $k$. 
 
Each bus $k$ has a fixed load $P_{k}^d + j Q_{k}^{d}$, where $P_{k}^{d} \geq 0$ is termed active power load, and  $- \infty \leq Q_{k}^{d} \leq + \infty$ is reactive power load; and lower $V_{k}^{\min} \geq 0$ and upper $V_{k}^{\max} \geq 0$ voltage limits.  For each branch $\{k,m\}$ we are given a thermal limit $0 \leq U_{km} \leq +\infty$, and maximum angle-difference $|\Delta_{km}| \leq \pi$. Thus, the goal is to find a voltage magnitude $|V_k|$ and phase angle $\theta_k$ at each bus $k$, and active $P^{g}$ and reactive $Q^g$ power generation for every generator $g$, so that power is transmitted by the network so as to satisfy active $P^{d}$ and reactive $Q^{d}$ power demands at minimum cost.  Using the so-called \textit{polar representation} we obtain the following nonlinear optimization problem:

\begin{subequations}\label{AC:firstformulation}
\begin{align} 
[\text{ACOPF}]: \hspace{5em}&\min \hspace{2em} \sum_{k \in \mathcal{G}} F_{k}(P_{k}^{g}) \label{AC:theobjective} \\
    \text{subject to:}\hspace{5em}& \nonumber\\
    \forall \, k \in \mathcal{B},\hspace{4em}& \nonumber\\
    \sum_{\{k,m\} \in \delta(k)} &P_{km} = \sum_{\ell \in \mathcal{G}_{k}} P_{\ell}^{g} - P_{k}^{d} \label{AC:activepowerbal}\\
    \sum_{\{k,m\} \in \delta(k)} &Q_{km} = \sum_{\ell \in \mathcal{G}_{k}} Q_{\ell}^{g} - Q_{k}^{d} \label{AC:reactivepowerbal}\\
    \forall \{k,m\} \in \mathcal{E},\hspace{3em}& \nonumber\\
    P_{km} = G_{kk}|V_{k}|^{2} &+ G_{km} |V_{k}| |V_{m}| \cos(\theta_{km}) \nonumber \\
    &+ B_{km} |V_{k}| |V_{m}| \sin(\theta_{km}) \label{AC:def_from_activepower} \\
    P_{mk} = G_{mm}|V_{m}|^{2} &+ |V_{k}| |V_{m}|G_{mk}\cos(\theta_{km}) \nonumber \\
    &-  B_{mk} |V_{k}| |V_{m}| \sin(\theta_{km}) \label{AC:def_to_activepower} \\
    Q_{km} = - B_{kk}|V_{k}|^{2} &+ |V_{k}| |V_{m}|B_{km}\cos(\theta_{km}) \nonumber \\ 
    &- G_{km} |V_{k}| |V_{m}| \sin(\theta_{km}) \label{AC:def_from_reactivepower}\\
    Q_{mk} = - B_{mm}|V_{m}|^{2} &+ |V_{k}| |V_{m}|B_{mk}\cos(\theta_{km}) \nonumber \\
    &+  G_{mk} |V_{k}| |V_{m}| \sin(\theta_{km}) \label{AC:def_to_reactivepower} \\
    \forall k \in \mathcal{G}: 
    P^{\min}_{k} \leq P_{k}^g \leq &P_{k}^{\max}, \, Q^{\min}_{k} \leq Q_{k}^g \leq Q_{k}^{\max}  \label{AC:genpowerlimit}\\
    \forall k \in \mathcal{B},
    \hspace{4em} &V^{\min}_{k} \leq v_{k} \leq V^{\max}_{k} \label{AC:voltlimit}\\
    \forall \{k,m\} \in \mathcal{E}, | \theta_{km} | &\leq \bar{\Delta}_{km}, \,\, \theta_{km} = \theta_{k} - \theta_{m} \label{AC:maxanglediff}\\
    \max \left\{P_{km}^{2} + Q_{km}^{2}\right. &, \left. P_{mk}^{2} + Q_{mk}^{2} \right\}  \leq U^{2}_{km} \label{AC:capacity}
\end{align}
\end{subequations}


In the above formulation, the physical parameters of each branch $\{k,m\} \in \mathcal{E}$ are described by
\begin{equation*}
    Y_{\{km\}} := \begin{pmatrix}
    G_{kk} + j B_{kk} & G_{km} + j B_{km} \\
    G_{mk} + j B_{mk} & G_{mm} + j B_{mk},
    \end{pmatrix}
\end{equation*}
which is the (complex) admittance matrix of branch $\{k,m\}$. These parameters model in~\eqref{AC:def_from_activepower}-\eqref{AC:def_to_reactivepower} physical active and reactive power flows. Inequalities~\eqref{AC:maxanglediff} -\eqref{AC:capacity} amount to flow capacity constraints, and inequalities~\eqref{AC:genpowerlimit}-\eqref{AC:voltlimit} impose operational limits on power generation and voltages. Constraints~\eqref{AC:activepowerbal}-\eqref{AC:reactivepowerbal} impose active and reactive power balance; the left-hand side represents power injection at bus $k \in \mathcal{B}$, while the right-hand side represents net power generation (generation minus demand) at bus $k$. Finally, for each generator $k \in \mathcal{G}$, it is customary to assume the functions $F_{k} : \mathbb{R} \to \mathbb{R}$ in the objective~\eqref{AC:theobjective} are convex piecewise-linear or convex quadratic. 

We remark that, often, constraint \eqref{AC:maxanglediff} is not present, and, when explicitly given, concerns angle limits $\bar{\Delta}_{km}$ that are \textit{small} (smaller than $\pi/2$). Under such assumptions there are equivalent ways to restate \eqref{AC:maxanglediff} involving the arctangent function and other variables present in the formulation (the same applies to convex relaxations).  

Please refer to the surveys~\cite{c13},~\cite{c14} for alternative, but equivalent, ACOPF formulations.

\subsection{Prior work}

We briefly review the very large literature on convex \textit{relaxations for ACOPF}. 

The simplest relaxations use, a starting point, a rectangular formulation of the ACOPF problem (rather than the polar setup described above) yielding a QCQP (quadratically constrained quadratic program) and rely on the well-known McCormick \cite{c15} reformulation to linearize bilinear expressions. This straightforward relaxation has long been known to provide very weak bounds. 

The SOC relaxation introduced in~\cite{c7}, which is widely known as the Jabr relaxation (see next subsection), has had significant impact due to its effectiveness as a lower bounding technique. While on the one hand the SOC relaxation is strong, it also yields formulations that, in the case of large ACOPF instances, are very challenging even for the best solvers. A wide variety of techniques have been proposed to strengthen the Jabr relaxation. In~\cite{c8} \emph{arctangent constraints} are associated with cycles, with the goal of capturing the relationship between the additional variables in the Jabr relaxation and phase angles -- these are formulated as bilinear constraints, and then linearized via McCormick~\cite{c15} inequalities. The two other strengthened SOC formulations proposed in~\cite{c8} add polyhedral envelopes for \emph{arctangent functions}, and dynamically generate semi-definite cuts for cycles in the network. \cite{c9} developed the Quadratic Convex (QC) relaxation, which corresponds to the Jabr relaxation strengthened with polyhedral envelopes for sine, cosine and bilinear terms appearing in the power flow definitions \eqref{AC:def_from_activepower}-\eqref{AC:def_to_reactivepower}. \cite{c16} proposes a \emph{minor-based} formulation for ACOPF (which is a reformulation of the rank-one constraints in the semidefinite programming formulation for ACOPF~\cite{c13}), which is relaxed to generate cutting-planes improving on the tightness of the Jabr relaxation. 

A semidefinite programming relaxation based on the \emph{Shor relaxation}~\cite{c17} is presented in~\cite{c18}. This formulation is at least as tight as the Jabr relaxation at the expense of even higher computational cost~\cite{c16}. 
Overall, experiments for all of these nonlinear relaxations have been limited to small and medium-sized cases, in part because SOCP solution technology may not be sufficiently mature.



Next we review linear relaxations for ACOPF. \cite{c19,c20} introduces the so-called active-power loss linear inequalities which state that on any branch the active power loss is nonnegative.  The resulting relaxation, which we term the linear-loss-relaxation, is shown to yield good lower bounds. 


In a similar same vein, \cite{c21} propose the network flow and the copper-plate relaxations. The network flow relaxation amount to the linear loss-relaxation with additional sparse linear inequalities that lower bound net reactive power losses in appropriate cases. Moreover, the copper-plate relaxation is obtained from the network flow relaxation by neglecting the power flow equations entirely via aggregation of all active and reactive power injections in the network. Along these lines \cite{c22} provide a relaxation which enforces a (valid) linear relationship between active and reactive power losses by relaxing linear combinations of \eqref{AC:def_from_activepower}-\eqref{AC:def_to_reactivepower}. The technique in ~\cite{c19} $\epsilon$-approximates the products of continuous variables (arising from the rectangular formulation of ACOPF~\cite{c13}), to arbitrary precision, using binary expansions and McCormick inequalities. This process yields a mixed integer linear $\epsilon$-approximation for ACOPF. Another linear $\epsilon$-approximation, which is based on the Jabr relaxation, is used in~\cite{c23}. Their main contribution is using the SOC linear approximation developed in~\cite{c24} which requires $O(k_{\ell} \log(1/\epsilon) )$ linear constraints and variables to $\epsilon$-approximate a conic constraint of row size $k_\ell$. Moreover, ~\cite{c25,c26} propose successive linear programming (SLP) algorithms for finding locally optimal AC solutions. One of the algorithms in~\cite{c26} is an SLP method focusing on the Jabr relaxation, and thus yielding a linear relaxation for ACOPF. We remark that the well-known Direct Current Optimal Power Flow (DCOPF) may prove a poor approximation to ACOPF in the sense that AC feasible solutions might not be $\epsilon$-feasible for DCOPF~\cite{c27} for arbitrary small $\epsilon > 0$.



We refer the reader to the  surveys~\cite{c13,c14,c28} for additional material on convex relaxations for ACOPF.

\subsection{Two Convex Relaxations for ACOPF}

\subsubsection{The Jabr SOCP}

A well-known convex relaxation of ACOPF is the \emph{Jabr relaxation}~\cite{c7}. It linearizes the power flow definitions \eqref{AC:def_from_activepower}-\eqref{AC:def_to_reactivepower} using $|\mathcal{B}| + 2 |\mathcal{E}|$ additional variables and adds $|\mathcal{E}|$ rotated-cone inequalities. A simple derivation is as follows: For any line $\{k,m\} \in \mathcal{E}$, we define $v_{k}^{(2)} := |V_{k}|^2$, $c_{km} := |V_{k}| |V_{m}| \cos(\theta_{k} - \theta_{m})$, $s_{km} := |V_{k}| |V_{m}| \sin(\theta_{k} - \theta_{m})$. Clearly we have the following valid (non-convex) quadratic relation
\begin{equation}
    c_{km}^{2} + s_{km}^{2} = v_{k}^{(2)} v_{m}^{(2)},
\end{equation}
which in~\cite{c7} is relaxed into the (convex) inequality
\begin{equation}\label{jabrineq}
    c_{km}^{2} + s_{km}^{2} \leq v_{k}^{(2)} v_{m}^{(2)}.
\end{equation}
This is a rotated-cone inequality hence it can be represented as a second-order cone constraint. Note that the newly defined variables $v_{k}^{(2)}$, $c_{km}$, and $s_{km}$ can be used to represent the power flow equations  in \eqref{AC:def_from_activepower}-\eqref{AC:def_to_reactivepower} as, $\forall \{k,m\} \in \mathcal{E}$,
\begin{subequations} \label{JABR:flows}
\begin{align}
    P_{km} &= G_{kk}v_{k}^{(2)} + G_{km}c_{km} +  B_{km} s_{km} \label{SOCP:def_to_activepower}\\
    P_{mk} &= G_{mm}v_{m}^{(2)} + G_{mk}c_{km} -  B_{mk} s_{km} \label{SOCP:def_from_activepower} \\
    Q_{km} &= - B_{kk}v_{k}^{(2)} + B_{km}c_{km} -  G_{km} s_{km} \label{SOCP:def_to_reactivepower}\\
    Q_{mk} &= - B_{mm}v_{m}^{(2)} + B_{mk} c_{km} +  G_{mk} s_{km}  \label{SOCP:def_from_reactivepower}
\end{align}
\end{subequations} 
\noindent In summary, the Jabr relaxation can be obtained from the formulation \eqref{AC:firstformulation} by (i) adding the $c_{km}, s_{km}, v^{(2)}_k$ variables,
(ii) replacing \eqref{AC:def_from_activepower}-\eqref{AC:def_to_activepower} with \eqref{JABR:flows}, and (iii) adding constraint \eqref{jabrineq}.\footnote{We stress that the definitions of $v_{k}^{(2)}$, $c_{km}$, and $s_{km}$ are \textit{not} added.}

\subsubsection{The i2 SOCP}

Recall that complex power injected into branch $\{k,m\} \in \mathcal{E}$ at bus $k \in \mathcal{B}$ is defined by
\begin{equation*}
    S_{km} := V_{k} I^{*}_{km},
\end{equation*}
hence, $|S_{km}|^{2} = |V_{k}|^{2} | I_{km} |^{2}$ holds. Moreover, since complex power can be decomposed into active and reactive power as $S_{km} = P_{km} + j Q_{km}$, and recalling that $v_{k}^{(2)} := |V_{k}|^{2}$ while denoting $i_{km}^{(2)} := |I_{km}|^{2}$, we have
\begin{equation}\label{fundamentalpowereq}
    P_{km}^{2} + Q_{km}^{2} = v_{k}^{(2)} i_{km}^{(2)}.
\end{equation}
By relaxing the equality \eqref{fundamentalpowereq} we obtain the rotated-cone inequality~\cite{c9,c29}
\begin{equation}\label{i2ineq}
    P_{km}^{2} + Q_{km}^{2} \leq v_{k}^{(2)} i^{(2)}_{km}.
\end{equation}
Since the variable $i_{km}^{(2)}$ can be defined linearly in terms of $v_{k}^{(2)}$, $v_{m}^{(2)}$, $c_{km}$, and $s_{km}$, we obtain an alternative SOC relaxation. This formulation, which we call the \emph{i2 relaxation}, is comprised by \eqref{AC:theobjective}-\eqref{AC:reactivepowerbal}, \eqref{JABR:flows}, the linear definition of $i_{km}^{(2)}$~\eqref{appendix:i2def_eq} and the rotated-cone inequalities~\eqref{i2ineq}, c.f. \ref{appendix:i2def}.

It is known~\cite{c30,c9,c31} that the systems defined by each branch's $\{k,m\}$ linearized power flows \eqref{JABR:flows} with its corresponding Jabr inequality~\eqref{jabrineq}, and on the other hand, branch's $\{k,m\}$ linearized power flows \eqref{JABR:flows} with the rotated-cone inequality~\eqref{i2ineq} and the linear definition of $i^{(2)}_{km}$~\ref{appendix:i2def}, are equivalent. In other words, for each feasible solution to one system there is a feasible solution to the other one. It must be noted though that in terms of the \emph{complete} formulations, equivalence always holds true if $i^{(2)}$ is not upper bounded.

\begin{proposition}\label{prop:i2_stronger_than_jabr}
    The Jabr and the i2 relaxations are equivalent if $i^{(2)}$ is not upper bounded, and otherwise the i2 relaxation can be strictly stronger.
\end{proposition}
\begin{proof}
    Sufficiency was proven in~\cite{c9}. For an example where the i2 relaxation is strictly stronger than the Jabr relaxation see Appendix~\ref{appendix:i2strongerjabr}.
\end{proof}

Our computational experiments corroborate this fact; we have found that linear outer-approximation cuts for the rotated-cone inequalities \eqref{jabrineq} and \eqref{i2ineq} have significantly different impact in lower bounding ACOPF (c.f. \ref{subsection:basicalgorithm}).



\section{Our work}

In this paper we use a dynamically generated linearly-constrained relaxation as a lower bounding procedure for ACOPF. We introduce a few concepts from the integer programming community. 

Given a set $X$ in $\mathbb{R}^{n}$, we say that a convex inequality $g(x) \leq d$ is \emph{valid} for $X$ if for every $x$ in $X$ $g(x) \leq d$ holds. For a set $R$ in $\mathbb{R}^{n}$, usually $R \supseteq X$ is a relaxation of $X$, we say that $c^{\top} x \leq d$ is a (linear) \emph{cut} for $X$ (relative to $R$) if the inequality is valid for $X$ but not for $R$. A (linear) cutting-plane algorithm~\cite{c32} for a set $X$ is an iterative procedure in which, starting from an initial (linear) relaxation, in every round (linear) cuts are added to outer-approximate a set $X$. In general, these cuts are computed using a solution $\overline{x}$ to the relaxation at the current round and valid inequalities which are violated by $\overline{x}$. In this paper, our target set $X$ is the i2 relaxation of ACOPF.

To justify the use of our methodology we note that direct solution of the Jabr and i2 relaxations of ACOPF, for large instances, is computationally prohibitive and often results in non-convergence (c.f. Tables~\ref{table:jabrsocps},~~\ref{table:i2socps},~\ref{table:ws_loads_1_1} and~\ref{table:ws_onelineoff}). Empirical evidence further shows that outer-approximation of the rotated-cone inequalities (in either case) requires a large number of cuts in order to achieve a tight relaxation value. Moreover, employing such large families of cuts yields a relaxation that, while linearly constrained, still proves challenging -- both from the perspective of running time and numerical tractability. Nonetheless, a characteristic feature of our iterative procedure is its robustness to potentially suboptimal termination of the oracle used to solve the LPs or convex QPs; \emph{independent} of the quality of the primal solution obtained, our linear cuts will always be valid.

However, as we show, adequate cut management proves successful, yielding a procedure that is (a) rapid, (b) numerically stable, and (c) constitutes a very tight relaxation (c.f. Tables~\ref{table:ws_loads_1_1} and~\ref{table:ws_onelineoff}).  The critical ingredients in this procedure are: (1) quick cut separation; (2) appropriate violated cut selection; and (3) effective dynamic cut management, including rejection of \textit{nearly-parallel} cuts and removal of \textit{expired} cuts, i.e., previously added cuts that are slack (cf.~\ref{subsection:basicalgorithm}).

Our procedure possesses efficient warm-starting capabilities -- this is a central goal of our work. Previously computed cuts, for some given instance, can be re utilized and loaded into new runs of \emph{related} instances, hence leveraging previous computational effort. It is worth noting that this reoptimization feature stands in sharp contrast to what is possible using nonlinear (convex) solvers. In~\ref{subsubsection:experiments_warmstarts} we justify the validity of this feature and Tables~\ref{table:ws_loads_1_1} and~\ref{table:ws_onelineoff} summarize extensive numerical evidence on its performance relative to solving the SOCPs `from scratch'. We remark that adequate cut management is what makes possible this feature for large ACOPF instances.

\subsection{Cuts}\label{subsection:cuts}

In this subsection we present a theoretical justification for using an outer-approximation cutting-plane framework on the Jabr and i2 relaxations, as well as computationally efficient cut separation procedures. We also give brief intuition on the complementarity of the Jabr and i2 outer-envelope cuts.

\subsubsection{Losses and Outer-Envelope Cuts}
For transmission lines with $G_{kk} > 0 > G_{km} = G_{mk} \geq - G_{kk}$ and $B_{km} = B_{mk}$, in particular lines with no transformer nor shunt elements, active-power loss inequalities are implied by the Jabr inequalities, and also by the definition of the $i^{(2)}$ variable. We remark that if negative losses are present, then total generation is smaller than total loads -- effectively, the negative losses amount to a source of free generation and directly contribute to a lower objective value for ACOPF that should not be feasible. A detailed discussion on losses and ACOPF relaxations is given in~\cite{bienstock15} and~\cite{cc}. We begin with two simple technical observations.

First, consider a (generic) rotated cone inequality 
\begin{equation}\label{eq:genericrotated}
x^{2} + y^{2} \leq wz,
\end{equation}
which is equivalent to $(2x)^{2} + (2y)^{2} \leq (w+z)^{2} - (w-z)^{2}$. Hence,
\begin{equation}
    x^{2} + y^{2} \leq wz \iff ||(2x,2y,w-z)^{\top}||_{2} \leq w+z. \label{eq:rotatedrewrite}
\end{equation}
Next, let $\lambda \in \mathbb{R}^3$ satisfy $|| \lambda||_2 = 1$.  Then, by \eqref{eq:rotatedrewrite},
\begin{align}
(2x,2y,w-z) \lambda &\leq 
||(2x,2y,w-z)^{\top}||_{2} \, || \lambda||_2 \nonumber \\
&\leq w+z. \label{eq:genericouter}
\end{align}
Inequality \eqref{eq:genericouter} provides a generic recipe to obtain outer-envelope inequalities for the rotated cone \eqref{eq:genericrotated}. As a result of these developments,  we have:
\begin{proposition}\label{prop:jabrouter}
    For a transmission line $\{k,m\} \in \mathcal{E}$ with $G_{kk} > 0 > G_{km} = G_{mk} \geq - G_{kk}$ and $B_{km} = B_{mk}$, the Jabr inequality $c_{km}^{2} + s_{km}^{2} \leq v_{k}^{(2)} v_{m}^{(2)}$ implies, as an outer envelope approximation inequality, that $P_{km} + P_{mk} \geq 0$.
\end{proposition}
\begin{proof}
    See Appendix~\ref{appendix:jabrouter}.
\end{proof}

Moreover, it is known that for transmission lines with no transformers nor shunt elements the definition of the variable $i^{(2)}$ implies the active-power loss inequalities~\cite{c30,c31}.

 
\subsubsection{Two Simple Cut Procedures}\label{subsubsection:separation}
The following proposition and corollary give us an inexpensive computational procedure for separating the rotated-cone inequalities
\begin{equation}\label{rotatedconeinequalities}
    c_{km}^{2} + s_{km}^{2} \leq v_{k}^{(2)} v_{m}^{(2)}, \hspace{2em} P_{km}^{2} + Q_{km}^{2} \leq v_{k}^{(2)} i_{km}^{(2)}.
\end{equation}

\begin{proposition}\label{proposition:projection}
    Consider the second-order cone $C := \{(x,s) \in \mathbb{R}^{n} \times \mathbb{R}_{+} \, : \, ||x||_{2} \leq s\}$. Suppose $(\overline{x},\overline{s}) \notin C$ with $\overline{s} > 0$. 
    Then the cut for $C$ which achieves the maximum violation by $(\overline{x},\overline{s})$ is given by $\overline{x}^{\top} x  \leq s || \overline{x} ||$.
\end{proposition}
\begin{proof}
    See Appendix~\ref{appendix:proofprojection}.
\end{proof}

\begin{corollary}\label{cor:cuts}
    Let $C := \{ (x,y,w,z) \in \mathbb{R}^{2} \times \mathbb{R}^{2}_{+} \, : \, x^{2} + y^{2} \leq wz \} \subseteq \mathbb{R}^{4}$ and suppose that $(\overline{x},\overline{y},\overline{w},\overline{z}) \notin C$ where $\overline{w}+\overline{z}>0$. The cut which achieves the maximum violation by $(\overline{x},\overline{y},\overline{w},\overline{z})$ is given by
    \begin{align}
        (4\overline{x})^{\top}x + (4\overline{y})^{\top}y &+ ((\overline{w}-\overline{z})-n_{0})^{\top} w \nonumber \\
        &+ (-(\overline{w}-\overline{z})-n_{0})^{\top}z \leq 0,
    \end{align}
    where $n_{0} := ||(2\overline{x},2\overline{y},\overline{w}-\overline{z})^{\top}||$.
\end{corollary}
\begin{proof}
     Rewriting the rotated-cone inequality as \eqref{eq:rotatedrewrite} and a direct application of Proposition~\ref{proposition:projection} gives us the desired separating hyperplane.
\end{proof}

Finally, we present a proposition which gives us a simple procedure for computing linear cuts for violated thermal-limit inequalities 
\begin{equation}\label{thermallimit}
    P_{km}^{2} + Q_{km}^{2} \leq U_{km}^{2}.
\end{equation}

\begin{proposition}\label{prop:thermalcuts}
    Consider the Euclidean ball in $\mathbb{R}^{2}$ of radius $r$, $S_{r}:= \{ (x,y) \in \mathbb{R}^{2} \, : \, x^{2} + y^{2} \leq r^{2}\}$, and let $(\overline{x},\overline{y}) \notin S_{r}$. Then the cut that attains the maximum violation by $(\overline{x},\overline{y})$ is given by
    \begin{equation}
        (\overline{x})^{\top}x + (\overline{y})^{\top}y \leq r ||(\overline{x},\overline{y})^{\top}||.
    \end{equation}
\end{proposition}
\begin{proof}
    See Appendix~\ref{appendix:proofthermalcuts}.
\end{proof}

\subsubsection{On the Complementarity of the Jabr and i2 cuts}\label{subsubsection:complementarity}
If $\{k,m\}$ is a transmission line with no transformer nor shunt elements, then
        \begin{equation}\label{i2simpleline}
            i_{km}^{(2)} = \left( \frac{1}{r_{km}^{2} + x_{km}^{2}} \right) \left( v_{k}^{(2)} + v_{m}^{(2)} - 2  c_{km}  \right)
        \end{equation}
where $r_{km}$ and $x_{km}$ denote line's $\{k,m\}$ resistance and reactance (see~\ref{appendix:i2def}). Suppose that $i_{km}^{(2)}$ is upper-bounded by some constant $H_{km}$ and that the line $\{k,m\}$ has a small resistance, e.g., on the order of $10^{-5}$ (p.u.). Since $x_{km}$ is usually an order of magnitude larger than $r_{km}$, the coefficient $(r_{km}^{2} + x_{km}^{2}) H_{km}$ can be fairly small, hence we have
\begin{equation}\label{surfacejabrcone}
    v_{k}^{(2)} + v_{m}^{(2)} - 2 c_{km} \leq (r_{km}^{2} + x_{km}^{2}) H_{km} \approx 0
\end{equation}
Since $v_{k}^{(2)} +  v_{m}^{(2)}  - 2 c_{km} \geq 0$ is a Jabr outer-envelope cut (c.f. proof Proposition~\eqref{prop:jabrouter}), inequality~\eqref{surfacejabrcone} is enforcing our solutions to be on the surface of the rotated-cone $c_{km}^{2} + s_{km}^{2} \leq v_{k}^{(2)} v_{m}^{(2)}$. 

\subsection{Basic Algorithm and Cut Management}\label{subsection:basicalgorithm}

In what follows we describe our cutting-plane algorithm. First we define a linearly constrained base model $M_{0}$ as follows:
\begin{subequations}\label{basemodel}
\begin{align}
[M_{0}]: &\hspace{2em}\min \hspace{2em} \sum_{k \in \mathcal{G}} F_{k}(P_{k}^{g}) \\
    \text{subject to:}& \nonumber \\
    &\text{constraints } \eqref{AC:activepowerbal},\eqref{AC:reactivepowerbal},\eqref{JABR:flows}, \eqref{AC:genpowerlimit}, \eqref{AC:voltlimit}
\end{align}
\end{subequations}
In other words, we consider the linearized power flow equations of the Jabr SOCP and all of the linear constraints in~\eqref{AC:firstformulation}. 

In every round of our iterative procedure, linear constraints will be added to and removed from $M_{0}$. The exact manner in how this will be done is described below. We will denote by $M$ our dynamic relaxation at some iteration of our cutting-plane algorithm. 


\begin{algorithm}\label{thealgorithm}
\caption{Cutting-Plane Algorithm}
\begin{algorithmic}[1]
\Procedure {Cutplane}{}
\State Initialize $r \gets 0$, $M \gets M_{0}$, $z_{0} \gets + \infty$
\While{$t < T$ and $r < T_{ftol}$}
\State $z \gets \min M$ and $\bar{x} \gets \text{argmin} \, M$
\State Check for violated inequalities by solution $\overline{x}$
\State Sort inequalities by violation
\State Compute cuts for the most violated inequalities
\State Add cuts if they are not $\epsilon$-parallel to cuts in $M$
\State Drop cuts of age $\geq T_{age}$ whose slack is $\geq$ $\epsilon_{j}$
\If{$z - z_{0} < z_{0} \cdot \epsilon_{ftol}$}
\State $r \gets r+1$
\Else
\State $r \gets 0$
\EndIf
\State $z_{0} \gets z$
\EndWhile
\EndProcedure
\end{algorithmic}
\end{algorithm}

Given a feasible solution $\bar{x}$ to $M$, and letting $f_{km}(x) \leq 0$ be some valid convex inequality~\eqref{rotatedconeinequalities} or~\eqref{thermallimit}, our measure of \emph{cut-quality} is the amount ${\max \{ f_{km}(\overline{x}), 0 \}}$ by which the solution $\overline{x}$ violates the valid convex inequality. Let $\epsilon > 0$, then for each type $\tau \in \{ \text{Jabr}, \text{i2}, \text{limit} \}$ of inequality, i.e., Jabr and i2 rotated-cones and thermal limits, we sort the branches $\{k,m\}$ from highest to lowest violation strictly greater than $\epsilon$, and pick as $\tau$-candidates branches, for which cuts will be added to $M$, the top $p_{\tau}$ (fixed parameter) percentage of the most violated branches. 

For each list of $\tau$-candidates, we compute cuts for the corresponding branches using the efficient cut procedures described in~\ref{subsection:cuts}. Candidate cuts will be rejected if they are \emph{too parallel} to incumbent cuts in $M$~\cite{c33,c34}. To be precise, given $\epsilon_{par} > 0$, we say that two linear inequalities $c^{t}x \leq 0$ and $d^{t} x \leq 0$ are \emph{$\epsilon_{par}$-parallel} if the cosine of the angle between their normal vectors $c$ and $d$ is strictly more that $1 - \epsilon_{par}$.


Finally, we describe a heuristic for \emph{cleaning-up} our formulation. For each added cut, we keep track of its current \emph{cut-age}, i.e., the difference between the current round and the round in which it was added to the relaxation. Then, in every iteration, if a cut $c^{\top}x \leq d$ has age greater or equal than a fixed parameter $T_{age} \in \mathbb{N}$, and it is \emph{$\epsilon$-slack}, i.e., $d - c^{\top} \overline{x} > \epsilon$, then it is dropped from $M$.

In addition to $M_{0}$ and the parameters $p_{\tau}, \epsilon, \epsilon_{par}, T_{age}$, other inputs for our procedure are: a time limit $T>0$; the number of admissible iterations without sufficient objective improvement $T_{ftol} \in \mathbb{N}$; and a threshold for objective relative improvement $\epsilon_{ftol} > 0$.

\subsection{Computational Results}\label{subsection:experiments}

\begingroup
\setlength{\tabcolsep}{6pt} 
\begin{table*}
\caption{Cutting-Plane (Not Warm-Started)}
\centering
\begin{tabular}{ @{} l r r r r r r r r@{} }
\toprule
& \multicolumn{5}{c}{Cutting-Plane} & \multicolumn{2}{c}{Primal bound} & \\
\cmidrule(lr{.75em}){2-6} \cmidrule(lr{.75em}){7-8}
\multicolumn{1}{l}{Case} & Objective & Time (s) & Computed & Added & Rounds & Objective & Time (s) & \\
\midrule
9241pegase & 309221.81 & 378.82 & 135599 & 29875 & 23 & 315911.56 & 96.74 & \\
9241pegase-api & 6924650.57 & 277.32 & 128316 & 30230 & 21 & 7068721.98 & 73.85 & \\
9241pegase-sad & 6141202.28 & 386.51 & 113686 & 27273 & 21 & 6318468.57 & 33.92 & \\
9591goc-api & 1346373.10 & 187.26 & 87812 & 22469 & 22 & 1570263.74 & 42.85 & \\
9591goc-sad & 1055493.25 & 246.87 & 90153 & 20514 & 27 & 1167400.79 & 28.15 & \\
ACTIVSg10k & 2476851.62 & 132.16 & 60803 & 18183 & 19 & 2485898.75 & 76.71 & \\
10000goc-api & 2502026.03 & 147.12 & 73084 & 19666 & 24 & 2678659.51 & 23.46 & \\
10000goc-sad & 1387303.02 & 114.97 & 58984 & 18528 & 17 & 1490209.66 & 103.06 & \\
10192epigrids-api & 1849488.30 & 152.87 & 97921 & 24882 & 22 & 1977687.11 & 117.15 & \\
10192epigrids-sad & 1672819.53 & 185.02 & 95740 & 23726 & 23 & 1720194.13 & 23.74 & \\
10480goc-api & 2708819.18 & 200.48 & 114967 & 29805 & 21 & 2863484.4 & 38.71 & \\
10480goc-sad & 2287314.69 & 270.38 & 118122 & 28004 & 24 & 2314712.14 & 27.93 & \\
13659pegase & 379084.55 & 841.83 & 176962 & 37297 & 22 & 386108.81 & 1184.15 & \\
13659pegase-api & 9270988.77 & 326.57 & 147479 & 34390 & 19 & 9385711.45 & 44.43 & \\
13659pegase-sad & 8868216.24 & 301.87 & 130682 & 32662 & 19 & 9042198.49 & 42.08 & \\
19402goc-api & 2448812.41 & 440.67 & 213564 & 52388 & 22 & 2583627.35 & 87.33 & \\
19402goc-sad & 1954047.79 & 488.33 & 218291 & 49749 & 25 & 1983807.59 & 64.01 & \\
20758epigrids-api & 3042956.88 & 464.17 & 189436 & 46124 & 25 & 3126508.3 & 61.39 & \\
20758epigrids-sad & 2612551.03 & 379.36 & 180790 & 44624 & 24 & 2638200.23 & 58.11 & \\
24464goc-api & 2560407.12 & 471.14 & 226595 & 57162 & 22 & 2683961.9 & 533.03 & \\
24464goc-sad & 2605128.51 & 506.39 & 222908 & 55242 & 23 & 2653957.66 & 73.87 & \\
ACTIVSg25k & 5993266.85 & 592.39 & 156285 & 43851 & 28 & 6017830.61 & 56.69 & \\
30000goc-api & 1531110.84 & 464.16 & 142385 & 41840 & 24 & 1777930.63 & 134.71 & \\
30000goc-sad & 1130733.51* & 147.74 & 76546 & 76546 & 6 & 1317280.55 & 565.05 & \\
ACTIVSg70k & 16326225.66 & 1065.76 & 350572 & 123431 & 13 & 16439499.83 & 240.55 & \\
78484epigrids-api & 15877674.54 & 1007.99 & 556893 & 240576 & 10 & 16140427.68 & 1079.03 & \\
78484epigrids-sad & 15175077.19 & 1062.55 & 501202 & 313587 & 8 & 15315885.86 & 343.45 & \\
\bottomrule
\end{tabular}
\label{table:cuts}
\end{table*}
\endgroup


We ran all of our experiments on an Intel(R) Xeon(R) Linux64 machine CPU E5-2687W v3 $3.10$GHz with $20$ physical cores, $40$ logical processors, and $256$ GB RAM. We used three state-of-the-art commercial solvers: Gurobi version 10.0.1~\cite{c35}, Artelys Knitro version 13.2.0~\cite{c36}, and Mosek 10.0.43~\cite{c37} For the SOCP and ACOPF we wrote AMPL modfiles and we ran them with a Python 3 script. We note that unlike Gurobi and Knitro, Mosek does not detect that a constraint like $x^2 + y^2 \leq z^2$ or $x^2 + y^2 \leq wz$ is actually a conic constraint, therefore we had to reformulate the SOCP to a format Mosek-AMPL was able to read. Now we describe the parameter specifications for each solver.

\paragraph{Gurobi} We use the Gurobi's default homogeneous self-dual embedding interior-point algorithm (barrier method without \emph{Crossover}, and \emph{Bar Homogeneous} set to $1)$, and we set the parameter \emph{Numeric Focus} equal to $1$. Barrier convergence tolerance and absolute feasibility and optimality tolerances were set to $10^{-6}$. Since by default Gurobi assigns any available cores to use for parallel computing automatically, it was not necessary to specify the number of threads. 

\paragraph{Knitro} We use Knitro's default Interior-Point/Barrier Direct Algorithm, with absolute feasibility and optimality tolerances equal to $10^{-6}$. We used the linear solver HSL MA57 sparse symmetric indefinite solver, and the Intel Math Kernel Library (MKL) functions for Basic Linear Algebra Subroutines (BLAS), i.e., for basic vector and matrix computations. Moreover, we gave Knitro $20$ threads to use for parallel computing features. When solving the SOCPs, we explicitly told Knitro that the problem is convex. We note that for computing primal bounds, we also tried the non-deterministic linear solver HSL MA97 whenever Knitro with MA57 was not converging. 

\paragraph{Mosek} We use Mosek's default homogeneous and self-dual interior-point algorithm for conic optimization. We set the relative termination tolerance, as well as primal and dual absolute feasibility tolerances to $10^{-6}$. On the test platform we assigned to Mosek $20$ threads.

\vspace{1em}

Our cutting-plane algorithm is implemented in Python 3 and calls Gurobi 10.0.1 as a subroutine for solving an LP or convex QP. All of our reported experiments were obtained with the following parameter setup: $\epsilon = 10^{-5}$, $p_{Jabr} = 0.55$, $p_{i2} = 0.15$, $p_{limit} = 1$, $T_{age} = 5$, $\epsilon_{par} = 10^{-5} /2$, $\epsilon_{ftol} = 10^{-5}$, and $T_{ftol} = 5$. All of our codes and AMPL model files can be downloaded from www.github.com/matias-vm, as well as links to our reported solutions.

We report extensive numerical experiments on instances with at least 9000 buses from the following data sets:
the Pan European Grid Advanced Simulation and State Estimation (PEGASE) project~\cite{c38,c39}, ACTIVSg synthetic cases developed as part of the US ARPA-E GRID DATA research project~\cite{c40,c41}, and the largest instances from the Power Grid Library for Benchmarking AC Optimal Power Flow Algorithms~\cite{c42}. 

We set a time limit of $1,000$ seconds for all of our SOCP experiments. We did not set a time limit for computing ACOPF primal bounds, and for our cutting-plane algorithm we enforced the $1,000$ seconds time limit before \emph{starting} a new round (one iteration of our algorithm takes around $60-100$ seconds for the largest cases). The character $``-"$ denotes that the solver did not converge, while the string $``\text{TLim}"$ means that the solver did not converge within our time limit of $1,000$ seconds. By convergence we mean that the solver \emph{declares} to have obtained an \emph{optimal} solution, within the previously defined tolerances.
We remark though that Gurobi and Knitro provide control of absolute primal and dual feasibility tolerances~\cite{c35,c36}, while Mosek only allows controlling normalized (by the RHS of the constraints) primal and dual feasibility tolerances~\cite{c37}. The string $``\text{INF}"$ means that the instance was declared infeasible by the solver, while $``\text{LOC INF}"$ that the instance might be locally infeasible. Moreover, if Gurobi declares 
\textit{numerical trouble} while solving our LPs or convex QPs at some iteration of our algorithm, we report the objective value of the previous iteration followed by the character $``*"$. We also note that objective values and running times are reported with 2 decimal places.

We remark that, to the best of our knowledge, this is the first computational study which compares the performance of three leading commercial solvers on the Jabr SOCP using a common framework (AMPL). We evaluate the solvers on Jabr SOCP, and compare our warm-started formulations on this formulation instead of the i2 SOCP because Jabr is numerically better behaved from the solvers' perspective. Indeed, the definition~\eqref{appendix:i2def_eq} of the $i^{(2)}$ variables can involve very large coefficients (on linear inequalities), yielding a numerically challenging nonlinear relaxation for most of the solvers. We report on these numerical issues in subsection~\ref{subsubsection:solversperformance}.

\subsubsection{Non-Warm-Started Cut Computations}\label{subsubsection:cuts&socps}

The purpose of Table~\ref{table:cuts} is summarize information regarding cut computations for a substantial number of instances from the libraries described above. The first multicolumn ``Cutting-Plane" subsumes information regarding our cutting-plane procedure: its first column ``Objective" reports the objective of the last iteration of our algorithm, the second column ``Time (s)" reports the total running time (in seconds) of our method; the third column ``Computed" reports the number of cuts computed throughout the whole procedure; the fourth column ``Added" exhibits the total number of cuts in our linearly constrained relaxation at the last round (these are the cuts used to warm-start our relaxations, c.f. \ref{subsubsection:experiments_warmstarts}); and the fifth column ``Rnd" the number of rounds of cuts. Finally, the last multicolumn ``Primal bound" reports the objective value of a feasible solution to ACOPF and the amount of time (in seconds) it took the nonlinear solver Knitro to find it.

\begingroup
\setlength{\tabcolsep}{6pt} 
\begin{table*}
\caption{Solvers' Performance on Jabr SOCP}
\centering
\begin{tabular}{ @{} l r r r r r r r @{} }
\toprule
& \multicolumn{3}{c}{Objective} & \multicolumn{3}{c}{Time (s)} & \\
\cmidrule(lr{.75em}){2-4} \cmidrule(lr{.75em}){5-7}
\multicolumn{1}{l}{Case} & Gurobi & Knitro & Mosek & Gurobi & Knitro & Mosek & \\
\midrule
9241pegase & - & 309234.16 & - & 82.11 & 34.68 & 31.11 & \\
9241pegase-api & - & 6840612.84 & - & 116.32 & 23.39 & 72.29 & \\
9241pegase-sad & - & 6083747.85 & - & 111.05 & 26.01 & 75.99 & \\
9591goc-api & 1346480.71 & 1348107.89 & 1345869.72 & 38.25 & 23.74 & 36.60 & \\
9591goc-sad & 1055698.54 & 1058606.56 & 1054379.58 & 49.29 & 32.83 & 37.61 & \\
ACTIVSg10k & - & 2468172.93 & 2466666.10 & 40.18 & 21.48 & 26.08 & \\
10000goc-api & - & 2507034.94 & 2498948.00 & 48.63 & 35.19 & 30.13 & \\
10000goc-sad & 1387288.49 & 1388679.63 & 1386041.07 & 23.58 & 26.27 & 23.68 & \\
10192epigrids-api & - & 1849684.14 & 1848873.47 & 75.82 & 42.69 & 29.09 & \\
10192epigrids-sad & - & 1672989.96 & 1672534.72 & 83.85 & 28.33 & 28.63 & \\
10480goc-api & - & 2708973.58 & 2707828.26 & 75.94 & 27.21 & 56.82 & \\
10480goc-sad & - & 2286454.3 & 2285547.23 & 149.93 & 38.17 & 59.48 & \\
13659pegase & 379135.73 & 379144.11 & - & 33.61 & 43.26 & 34.92 & \\
13659pegase-api & - & 9198542.14 & - & 162.21 & 30.64 & 105.11 & \\
13659pegase-sad & 8826902.31 & 8826958.23 & 8787429.86 & 83.75 & 31.84 & 108.74 & \\
19402goc-api & - & 2449020.25 & 2447799.72 & 158.12 & 152.89 & 103.04 & \\
19402goc-sad & - & 1954331.70 & 1952550.06 & 203.56 & 155.89 & 104.88 & \\
20758epigrids-api & - & - & 3040421.02 & 143.99 & TLim & 93.46 & \\
20758epigrids-sad & - & - & 2610196.94 & 98.30 & TLim & 75.88 & \\
24464goc-api & 2548335.96 & - & 2558631.63 & 603.95 & TLim & 129.90 & \\
24464goc-sad & - & - & 2603525.46 & 333.50 & TLim & 128.50 & \\
ACTIVSg25k & 5956787.54 & 5964417.54 & 5955368.56 & 169.66 & 87.14 & 87.18 & \\
30000goc-api & - & 1531256.65 & 1529197.81 & 207.60 & 118.80 & 123.38 & \\
30000goc-sad & - & - & 1130868.71 & 191.22 & TLim & 84.90 & \\
ACTIVSg70k & - & 16221577.73 & 16217263.66 & 553.26 & 320.98 & 232.47 & \\
78484epigrids-api & - & - & - & 756.00 & TLim & 637.48 & \\
78484epigrids-sad & 15180775.21 & - & 15169401.54 & 463.17 & TLim & 601.04 & \\
\bottomrule
\end{tabular}
\label{table:jabrsocps}
\end{table*}
\endgroup

Overall, we see that our cut management heuristics permits us obtain very tight linearly constrained relaxations with a relatively small number of cuts - note that we could potentially add $3|\mathcal{E}|$ cuts per round (since for each transmission $\{k,m\}$ there are exactly three nonlinear valid convex inequalities~\eqref{rotatedconeinequalities} and \eqref{thermallimit} that could be violated). For instance, case ACTIVSg70k has $88207$ branches and after 10 rounds of cuts we end up keeping $123431$ out of the $350572$ linear cuts computed throughout the course of our algorithm. Therefore, fewer than $1.5$ of linear cuts per branch gives us a relaxation with \emph{optimality gap}\footnote{Given a primal bound of a minimization problem, we define the optimality gap of a relaxation of the given problem as $\frac{z_{p}-z_{r}}{z_{p}}$, where $z_{p}$ denotes the objective value of the primal bound and $z_{r}$ denotes the objective value of the relaxation.} equal to $0.69 \%$.

We remark that for some instances the objective value of our procedure can be higher than the objective value of the Jabr SOCP since our algorithm is outer-approximating the feasible region of the i2 SOCP (c.f. Proposition~\ref{prop:i2_stronger_than_jabr}).

 
We also note that 30000goc-sad was the only instance for which the solver ran into numerical trouble while solving the linearly constrained relaxation (indicated by the character $``*"$ next to the objective value). Therefore, the reported objective value corresponds to the previous iteration. Setting a more aggressive cut management heuristic, for instance decreasing $T_{age}$ from $4$ to $5$, gave us numerically more stable cuts and a better bound.

Finally, we obtained our primal bounds by running Knitro with a \emph{flat-start}, i.e., we provided as initial point voltages set to $1$ and $\theta_{km} = 0$.

\subsubsection{Solvers' Performance on Jabr SOCP}\label{subsubsection:solversperformance}


In Table~\ref{table:jabrsocps}, we observe that for the cases in which at least two solvers converge, the reported bounds for the Jabr SOCP agree on the first 3 most significant digits. These differences in bounds across the different solvers reflect how numerically challenging the given instances are. We remind the readers of the parameter choices that we made in order for the solvers to achieve termination -- which otherwise would often fail.

As we mentioned at the beginning of this section, the i2 SOCP is numerically even more challenging for the solvers than the Jabr SOCP. Indeed, in Table~\ref{table:i2socps} we can see that the solvers do struggle. We studied in detail some cases where Gurobi AMPL declared optimality, for example case ACTIVSg70k, and observed variable bound max violation (scaled) equal to $8.43$ as well as large primal and dual residuals ($0.0128$ and $3.25$, respectively). Moreover, we noticed inconsistent termination status for cases 10192epigrids-sad, 10480goc-api, 20758epigrids-sad, and 30000goc-sad on Gurobi and Gurobi through AMPL (Gurobi-AMPL) using the same model; Gurobi AMPL declares optimal termination for these instances while Gurobi does not. Cases for which we were able to identify \emph{low quality} solutions or inconsistencies have been denoted with the character ``\dag" next to their reported objective value in Table~\ref{table:i2socps}.



\subsubsection{Warm-Starts}\label{subsubsection:experiments_warmstarts}

In power engineering practice, it is often the case that a power flow problem (either in the AC or DC version) is solved on data that reflects a recent, and likely limited, update on a case that was previously handled. In power engineering language, a 'prior solution' was computed, and the problem is not solved 'from scratch.' In the context of our type of algorithm, this feature opens the door for the use of \textit{warm-started formulations}, i.e., the application of a cutting-plane procedure that leverages previously computed cuts to obtain sharp bounds more rapidly than `from scratch'. In this subsection we present this warm-starting feature of our algorithm; we justify its validity and show via numerical experiments its appealing lower bounding capabilities. 

The convex inequalities~\eqref{rotatedconeinequalities}, based on which we are dynamically adding cuts, do not depend on input data such as loads or operational limits. Any such inequality remains valid and can be used if the associated branch remains operational.  This will be our strategy, below.

We created two kinds of perturbed instances: a) Instances were the load of each bus was perturbed by a Gaussian $(\mu,\sigma) = (0.01 \cdot P_{d}, 0.01 \cdot P_{d})$, where $P_{d}$ denotes the original load, subject to the newly perturbed load being non negative; and b) instances were the transmission line which carries the largest amount of active power in an ACOPF solution is turned off. We note though that perturbed cases b) do change the structure of the network, since we are setting off the status of an active branch. Hence, when warm-starting type b) cases, we will skip any cuts associated to the branch which becomes inactive.

Tables~\ref{table:ws_loads_1_1} and~\ref{table:ws_onelineoff} summarize our warm-started experiments on perturbed instances from our data set in Table~\ref{table:cuts} and compare to solvers' performance on the Jabr SOCP. ``First Round" reports the objective value and running time of the relaxation $M_{0}$ loaded with the cuts computed in Table~\ref{table:cuts}, i.e., our warm-started relaxation.  ``Last Round" presents the objective value and running time of the last iteration of our cutting-algorithm (on the warm-started relaxation).  ``Jabr SOCP" and ``Primal bound" report, respectively, on the objective value and running time of the Jabr SOCP for the three solvers, and ACOPF primal solutions.

We stress the comparison between the running time for our first round, and the solvers' running time.

\paragraph{Loads perturbed -- Gaussian deltas} 
For most of the instances, our procedure proves very tight lower bounds in less than 25 seconds (``First Round" column). Judging by the time it takes Knitro to find a locally optimal primal bound and by the number of cases in which the solvers converge, running the SOCPs, it seems that these instances are overall more challenging than their unperturbed counterparts.

Our procedure also stands out in quickly lower bounding the largest cases. For instance, a very sharp bound for case ACTIVSg70k is obtained in 102.25 seconds, taking less than half of the time it takes the fastest SOCP solver to converge. Similar performance is achieved on the largest epigrids cases where our method is 3x to 5x faster.   

An interesting empirical fact is that our cuts are robust with respect to load perturbations. Indeed, our evidence shows that there is not a considerable improvement from the ``First Round" to the ``Last Round" objectives.
This means that the previously pre-computed cuts loaded to $M_{0}$ in the first iteration are accurately outer-approximating the SOC relaxations.

Moreover, our linearly constrained relaxations are able to prove infeasibility for case 9241pegase-api in 23.10 seconds while none of the three solvers were able to provide a certificate of infeasibility for the Jabr SOCP. Knitro required 1845.42 seconds to declare convergence to a locally infeasible solution. Similar results are obtained for case 24464ogc-api.

The only case were our method fails to provide a valid lower bound is case 30000goc-sad -- our minimization oracle reports numerical trouble and fails to provide a solution to our warm-started relaxation. This is not surprising since difficult numerical behavior was noticed when computing cuts for this case.

\paragraph{Transmission line with largest flow turned off}
Overall, our method achieves a similar performance on this set of perturbed instances as in a); sharp lower bounds are obtained in about 25 seconds for most of the cases.

For this data set, our method and all of the SOCP solvers are able to prove infeasibility relatively quickly. On the other hand, our method proves a lower bound for ACTIVSg70k relatively quickly in the first round, but fails to converge in the next round due to numerical trouble caused by the newly added cuts.

As in a), our warm-started formulation achieves a good performance on the largest epigrid cases --  bounds are sharp with respect to the SOC relaxations and it is at least 3x faster.

The only case were our method fails to provide a lower bound is case 30000goc-sad -- our minimization oracle reports numerical trouble and fails to provide a solution to our warm-started relaxation.

\begingroup
\setlength{\tabcolsep}{6pt} 
\begin{table*}
\caption{Solvers' Performance on i2 SOCP}
\centering
\begin{tabular}{ @{} l r r r r r r r @{} }
\toprule
& \multicolumn{3}{c}{Objective} & \multicolumn{3}{c}{Time (s)} & \\
\cmidrule(lr{.75em}){2-4} \cmidrule(lr{.75em}){5-7}
\multicolumn{1}{l}{Case} & Gurobi & Knitro & Mosek & Gurobi & Knitro & Mosek & \\
\midrule
10192epigrids-api & 1849683.44 & 1849684.2 & - & 37.44 & 19.4 & 30.74 & \\
10192epigrids-sad & 1672998.72\dag & 1672998.73 & - & 23.36 & 21.48 & 24.53 & \\
10480goc-api & 2709110.52\dag & 2709110.71 & - & 41.44 & 27.78 & 35.28 & \\
10480goc-sad & 2287736.73 & 2287715.33 & - & 41.11 & 28.46 & 28.48 & \\
13659pegase & 379142.67 & - & - & 52.12 & TLim & 36.49 & \\
13659pegase-api & 9287242.7 & 9287244.72 & - & 66.39 & 236.65 & 30.01 & \\
13659pegase-sad & 8878803.69 & - & - & 63.11 & TLim & 30.48 & \\
19402goc-api & 2449100.15\dag & 2449102.05 & - & 79.38 & 55.6 & 54.18 & \\
19402goc-sad & 1954367.11\dag & 1954367.2 & - & 180.97 & 59.46 & 82.55 & \\
20758epigrids-api & - & 3043275.95 & - & 79.03 & 64.17 & 56.02 & \\
20758epigrids-sad & 2612841.71\dag & 2612841.8 & - & 48.32 & 84.07 & 58.87 & \\
24464goc-api & 2560829.65\dag & - & - & 132.34 & TLim & 88.79 & \\
24464goc-sad & 2605532.65\dag & - & - & 74.43 & 916.1 & 65.41 & \\
ACTIVSg25k & 5994727.45 & - & - & 70.61 & TLim & 52.38 & \\
30000goc-api & 1531320.78 & 1531322.2 & - & 96.91 & 593.73 & 71.38 & \\
30000goc-sad & 1132242.88\dag & 1132256.94 & - & 78.0 & 325.11 & 74.61 & \\
ACTIVSg70k & 16333807.38\dag & - & - & 300.98 & TLim & 209.3 & \\
78484epigrids-api & 15882668.49 & 15882668.46 & 15882654.42 & 216.15 & 315.31 & 203.81 & \\
78484epigrids-sad & 15180792.15 & 15180792.0 & 15180763.6 & 250.43 & 376.82 & 222.17 & \\
\bottomrule
\end{tabular}
\label{table:i2socps}
\end{table*}
\endgroup

\begingroup
\setlength{\tabcolsep}{3pt} 

\begin{table*}
\caption{Warm-Started Relaxations, Loads perturbed Gaussian $(\mu,\sigma) = (0.01 \cdot P_{d}, 0.01 \cdot P_{d})$}
\centerline{
\begin{tabular}{ @{} l r r r r r r r r r r r r r @{} }
\toprule
& \multicolumn{4}{c}{Cutting-Plane} & \multicolumn{6}{c}{Jabr SOCP} &  \\
\cmidrule(l{0.5em}r{0.40em}){2-5} \cmidrule(l{0.5em}r{0.40em}){6-11}
& \multicolumn{2}{c}{First Round} & \multicolumn{2}{c}{Last Round} & \multicolumn{3}{c}{Objective} & \multicolumn{3}{c}{Time (s)} & \multicolumn{2}{c}{Primal bound} & \\
\cmidrule(l{0.5em}r{0.40em}){2-3} \cmidrule(l{0.5em}r{0.40em}){4-5} \cmidrule(l{0.5em}r{0.40em}){6-8} \cmidrule(l{0.5em}r{0.40em}){9-11} \cmidrule(l{0.5em}r{0.40em}){12-13} 
\multicolumn{1}{l}{Case} & Objective & Time (s) & Objective & Time (s) & Gurobi & Knitro & Mosek & Gurobi & Knitro & Mosek & Objective & Time (s) &  \\
\midrule
9241pegase    &     309288.32     &     13.78     &     309299.97     &    160.28    &   -   &    309302.67    &  -  &   73.12   &    32.21    &  36.04  & 315979.53 & 101.48 &      \\
9241pegase-api    &    INF      &     23.10     &     INF     &   23.10     &   -   &    -    &  -  &   134.53   &    TLim    &  72.96  & LOC INF & 1845.92 &  \\
9241pegase-sad    &     6153913.91     &    16.18      &    6154117.59      &    136.78    &   -   &    6096743.03    &  -  &   97.51   &    26.07    &  83.43  & 6333763.92 & 43.71 &  \\
9591goc-api    &     1343642.47     &     11.06     &     1343670.62    &    56.36    &   1343767.43   &    1345384.57    &  1343190.29  &   39.36   &    25.36    &  35.30  & 1571582.59 & 54.16 &  \\
9591goc-sad    &     1058124.48     &    12.62      &     1058157.44     &    65.37    &   1058337.76   &    1061275.83    &  1057323.31  &   51.85   &    34.04    &  37.52  & 1178895.53 & 29.53 &  \\
ACTIVSg10k    &    2475041.43      &      9.52    &     2475078.69     &    50.51     &   -   &    2466383.20    &  -  &   42.31   &    21.75    &  29.33  & 2484093.15 & 57.24 &  \\
10000goc-api    &   2502049.28       &    8.51      &     2502098.01    &    36.48     &   2501946.30   &    2507074.78    &  2499373.75  &   31.91   &    43.44    &  32.33  & LOC INF & 1677.21 &  \\
10000goc-sad    &     1388833.86     &    8.70      &     1388859.09     &    44.50     &   1388824.91   &    1390230.41    &  1387588.17  &   25.96   &    29.31    &  23.67  & 1493481.44 & 93.72 &  \\
10192epigrids-api    &     1848085.36     &    10.27      &    1848133.48      &    45.84    &   -   &    1848285.26    &  1847120.93  &   65.38   &    41.17    &  25.99  & LOC INF & 1458.35 &  \\
10192epigrids-sad    &     1672358.89     &    10.33      &     1672398.61     &    53.37    &   -   &    1672533.02    &  1671364.67  &   73.64   &    28.61    &  35.66  & 1717429.36 & 23.89 &  \\
10480goc-api    &      2704157.29    &    12.43      &     2704252.95     &    58.45    &   -   &    2704373.73    &  2703432.85  &   197.17   &    27.57    &  55.92  & 2868495.28 & 36.89 &  \\
10480goc-sad    &    2294908.37      &    12.81      &    2294990.69      &    70.93    &   -   &    2294080.35    &  2292830.56  &   185.22   &    35.90    &  58.31  & 2322198.81 & 27.34 &  \\
13659pegase    &     379742.62     &    60.74      &     379794.51     &    426.88    &   379799.37   &    379804.43    &  -  &   34.21   &    43.17    &  32.75  & 386765.25 & 370.23 &  \\
13659pegase-api    &     9253539.07     &    21.25      &     9253773.43     &    109.20    &   9181205.93   &    9181269.20    &  -  &   97.11   &    30.41    &  118.31  & 9368277.57 & 62.20 &  \\
13659pegase-sad    &     8865733.59     &    21.28      &     8865892.49     &    113.04    &   8824442.20   &    8824486.03    &  -  &   86.49   &    33.19    &  102.59  & 9039904.52 & 40.02 &  \\
19402goc-api    &     2452185.69     &    23.55      &     2452270.83     &    120.10    &   -   &    2452448.33    &  2451708.50  &   146.87   &    120.39    &  103.32  & LOC INF & 4440.99 &  \\
19402goc-sad    &    1956255.19      &    23.28      &     1956313.91     &    113.89    &   -   &    1956570.60    &  1955018.07  &   231.90   &    172.82    &  102.19  & 1986936.95 & 66.02 &  \\
20758epigrids-api    &     3043006.76     &     22.34     &    3043076.56      &    104.06    &   -   &    -    &  3032919.24  &   134.60   &    TLim    &  78.32  & LOC INF & 12425.89 &  \\
20758epigrids-sad    &     2610197.53     &     20.46     &     2610261.88     &    93.09    &   -   &    -    &  2608090.26  &   143.69   &    TLim    &  72.19  & 2635892.81 & 49.25 &  \\
24464goc-api    &     2561680.14     &   26.28      &      INF    &    50.38    &   -   &    LOC INF    &  -  &   223.07   &    573.37    &  118.6  & - & 19444.54 &  \\
24464goc-sad    &     2606391.76     &     26.78     &    2606473.78      &     133.54   &   -   &    -    &  2604708.86  &   423.12   &    TLim    &  128.84  & 2655942.01 & 72.48 &  \\
ACTIVSg25k    &     5988886.18     &    28.24      &     5989016.75     &    198.58    &   5952404.50   &    5960068.30    &  5949381.04  &   138.01   &    73.75    &  109.39  & 6013477.05 & 57.87 &  \\
30000goc-api    &    1527412.96      &   25.35       &   1527487.45   &    151.75    &   -   &    1528338.73    &  1525625.64  &   243.61   &    369.83    &  119.92  & LOC INF & 3407.47 &  \\
30000goc-sad    &    -      &     46.33     &     -     &     46.33     &   -   &    -    &  1132715.53  &   257.94   &    TLim    &  75.20  & 1318389.55 & 620.27 &  \\
ACTIVSg70k    &    16316572.42      &    102.25      &    16317886.35       &    536.51    &   -   &    16210682.53    &  16206290.43  &   498.80   &    309.56    &  229.07  & 16428367.50 & 243.84 &  \\
78484epigrids-api    &     15862318.24     &     115.76     &     15865624.98     &     883.93     &   -   &    -    &  15859950.52  &   757.64   &    TLim    &  642.24  & - & 8113.53 &  \\
78484epigrids-sad    &     15176866.00     &     151.77     &     15180592.27     &    1118.02    &   15182602.75   &    -    &  15174716.43  &   420.56   &    TLim    &  589.46  & 15316872.94 & 353.13 &  \\
\bottomrule
\end{tabular}
\label{table:ws_loads_1_1}
}
\end{table*}

\endgroup

\begingroup
\setlength{\tabcolsep}{3pt} 

\begin{table*}
\caption{Warm-Started Relaxations, Transmission Line with largest flow turned off}
\centerline{
\begin{tabular}{ @{} l r r r r r r r r r r r r r @{} }
\toprule
& \multicolumn{4}{c}{Cutting-Plane} & \multicolumn{6}{c}{Jabr SOCP} &  \\
\cmidrule(l{0.5em}r{0.40em}){2-5} \cmidrule(l{0.5em}r{0.40em}){6-11}
& \multicolumn{2}{c}{First Round} & \multicolumn{2}{c}{Last Round} & \multicolumn{3}{c}{Objective} & \multicolumn{3}{c}{Time (s)} & \multicolumn{2}{c}{Primal bound} & \\
\cmidrule(l{0.5em}r{0.40em}){2-3} \cmidrule(l{0.5em}r{0.40em}){4-5} \cmidrule(l{0.5em}r{0.40em}){6-8} \cmidrule(l{0.5em}r{0.40em}){9-11} \cmidrule(l{0.5em}r{0.40em}){12-13} 
\multicolumn{1}{l}{Case} & Objective & Time (s) & Objective & Time (s) & Gurobi & Knitro & Mosek & Gurobi & Knitro & Mosek & Objective & Time (s) &  \\
\midrule
9241pegase  &  INF  &  10.86  &  INF  &  10.86  &  INF  &  INF  &  INF  &  8.36  & 7.55   &  10.22  & INF & 8.37 &  \\
9241pegase-api  &  INF  &  7.55  &  INF  &  7.55  &  INF  &  INF  &  INF  &  7.92  &  8.02  &  10.22  & INF & 8.23 &    \\
9241pegase-sad  &  INF  &  7.33  &  INF  &  7.33  &  INF  &  INF  &  INF  &  8.14  & 8.10 &  10.31  & INF & 8.46 &  \\
9591goc-api  &  1346470.95  &  10.42  &  1346859.06  &  60.76  &  1346969.75  &  1348591.44  &  1346437.99  &  39.30  &  17.98  &  36.89  & 1395829.51 & 28.08 &  \\
9591goc-sad  &  1055823.53  &  11.51  &  1056267.57  & 101.64   &  1056447.48  & 1059382.10 &  1055501.31  &  45.09  & 35.41   & 37.18 & 1199276.44 & 29.90 &  \\
ACTIVSg10k  & 2477043.05   &  9.94  &  2477537.79  &  75.85  &  -  &  2468821.96  &  2466981.35  &  44.81  &  21.60  &  17.52  & LOC INF & 7092.45 &  \\
10000goc-api  &  2506671.15  &  8.06  & 2509971.69 &  46.10  & 2509846.00 &  2514991.16  &  2506236.75  &  31.03  &  33.95  &  32.15  & 2692320.35 & 23.28 &  \\
10000goc-sad  &  1387382.65  &  8.76  &  1387515.89  &  66.14  & 1387480.33   &  1388870.68  &  1386283.75  &  26.53  &  34.24  &  24.14  & 1506187.88 & 108.19 &  \\
10192epigrids-api  &  1849901.82  &  9.47  &  1850621.81  &  68.73  &  -  &  1850788.76  &  1849821.44  &  69.81  &  38.01  & 25.30  & 2021493.05 & 117.18 &  \\
10192epigrids-sad  &  1673575.50  &  11.08  &  1674274.99  &  74.49  &  -  &  1674417.21  &  1673564.57  &  69.91  &  43.91  &  28.54  & 1734014.50 & 24.11 &  \\
10480goc-api  &  2710040.46  &  11.33  &  2711100.23  &  73.85  &  -  &  2711224.27  &  2710520.15  &  95.40  &  27.99  &  56.58  & 2862699.50 & 225.06 &  \\
10480goc-sad  &  2288069.64  &  13.47  & 2288969.47 & 98.88 & - &  2288069.23  &  2286864.08  &  106.54  &  37.65  &  59.43  & 2318279.76 & 26.13 &  \\
13659pegase  &  379102.13  &  53.29  &  379163.58  & 199.96 &  379177.99  &  379182.14  &  -  &  33.18  &  345.83  & 31.90 & 386126.93 & 394.93 &  \\
13659pegase-api  &  INF  &  10.81  &  INF  & 10.81 &  INF  &  INF  &  INF  &  10.94  &  8.90  &  13.79  & INF & 11.55 &  \\
13659pegase-sad  &  INF  & 10.63 &  INF  & 10.63 &  INF  &  INF  & INF &  11.02  &  10.80  &  13.76  & INF & 11.73 &  \\
19402goc-api  &  2450110.09  &  23.93  &  2451621.60  &  171.31 &  -  &  2451793.39  &  2450488.01  &  154.42  &  132.03  &  104.58  & 2587915.50 & 403.20 &  \\
19402goc-sad  & 1954365.39 & 23.70 &  1954881.06  &  191.52  &  -  &  1955116.35  &  1953676.05  &  258.93  &  156.10  &  102.63  & 1985954.83 & 63.38 &  \\
20758epigrids-api  &  3043482.21  &  20.44  &  3044690.46  & 133.92  & -  &  -  &  3041974.74  &  112.78  &  TLim  & 96.19 & 3132571.31 & 52.82 &  \\
20758epigrids-sad  & 2612646.70 & 20.62 & 2612786.37 & 115.89 &  -  &  -  &  2610315.41  &  169.67  &  TLim  &  73.02  & 2638560.64 & 47.87 &  \\
24464goc-api  & 2560669.11 &  25.94  & 2561110.03 & 161.80 &  2550118.22  &  - &  2559240.55  &  440.81  &  TLim  &  118.05  & 2684708.93 & 1663.63 &  \\
24464goc-sad  & 2605179.75 & 26.98 & 2605369.03 & 166.81 &  -  &   2605474.23 &  2603609.34  &  564.27  &  74.69  &  124.19  & 2654344.45 & 76.39 &  \\
ACTIVSg25k  &  6045885.88  &  27.52  &  6048122.86  & 238.67 &  6009656.52  &  6018875.03  &  6009500.57  &  144.28  &  65.42  & 81.41 & LOC INF & 1634.38 &  \\
30000goc-api  &  1531110.55  &  25.02  & 1531159.95 & 130.30 &  -  &  1532013.41  &  1529195.12  &  195.74  &  135.71  &  119.77  & LOC INF & 3203.26 &  \\
30000goc-sad  &  - &  45.60  &  -  &  45.60  &  -  &  -  &  1130917.78  &  218.12  &  TLim  &  74.27  & 1324622.71 & 186.43 &  \\
ACTIVSg70k  &  16426522.74  &  98.51  &  16426522.74*  &  98.51  &  -  &  -  &  -  &  150.77  &  TLim  &  129.60  & LOC INF & 3160.81 &  \\
78484epigrids-api  &  15888353.48  &  104.32  &  15892229.11  & 916.88 &  15894055.55  &  -  &  15880422.78  &  322.22  &  TLim  &  625.74  & 16169740.92 & 2328.79 &  \\
78484epigrids-sad  &  15179882.22  &  149.65  &  15185980.69  &  1151.33  &  15188085.99  &  -  &  15182701.65  &  437.59  & TLim   &  594.84  & 15330674.69 & 272.41 &  \\
\bottomrule
\end{tabular}
\label{table:ws_onelineoff}
}
\end{table*}
\endgroup

\section{CONCLUSIONS AND FUTURE WORK} 

In this paper we present a fast (linear) cutting-plane method used to obtain tight relaxations for even the largest ACOPF instances, by appropriately outer-approximating the SOC relaxations. Our relaxations can be constructed and solved robustly and quickly via a cutting-plane algorithm that relies on proper cut management and leverages mature linear programming technology. 

The central focus on this paper concerns \textit{reoptimization}. We show that our procedure possesses efficient warm-starting capabilities -- previously computed cuts, for some given instance, can be re utilized and loaded into new runs of \emph{related} instances, hence leveraging previous computational effort. As a main contribution we demonstrate, through extensive numerical testing in medium to (very) large instances, that the warm-start feature for our cutting-plane algorithm yields tight and accurate bounds far faster than otherwise possible. It is worth noting that this capability stands in contrast to what is possible using nonlinear (convex) solvers.
 
We believe our work paves the way for promising new research directions. For instance, since our relaxations are linear they could be deployed for practical pricing schemes which could increase welfare and mitigate biasedness in price signals~\cite{c12}. Moreover, we believe our relaxation is a natural candidate to supersede the well-known DC linear approximation in harder problems such as the Unit-Commitment problem or Security Contrained ACOPF (SCOPF), hence it would be interesting to evaluate its performance on these challenging problems.


\section{ACKNOWLEDGMENTS}

We would like to thank Erling Andersen, Bob Fourer, Ed Klotz and Richard Waltz.  This work was supported by an ARPA-E GO award.

\section{APPENDIX}\label{appendix}

\subsection{i2 SOCP strictly stronger than Jabr SOCP}\label{appendix:i2strongerjabr}

An instance where the i2 SOCP is strictly stronger than the Jabr SOCP is case1354pegase. Knitro attains an optimal solution to the Jabr SOCP of value 74009.28 while an optimal solution to the i2 SOCP has value 74013.68. As a sanity check, we fixed, within tolerance $\pm 10^{-5}$, the solution to the Jabr SOCP in the i2 SOCP's formulation, and Gurobi declared the SOCP infeasible and provided the following Irreducible Inconsistent Subsystem (IIS):

\begin{align*}
      i^{(2)}_{549,5002} &= 11822.45384038167 \, v^{(2)}_{549} \\
    &+ 11822.45384038167 \, v^{(2)}_{5002} \\ 
    &- 23644.8888107824 \, c_{549,5002} \\
    &- 29.87235441454166 \, s_{549,5002} \\
     v^{(2)}_{549}  &\geq 1.209989999283128 \\
     v^{(2)}_{5002} &\geq  1.19745626014781 \\
    c_{549,5002} &\leq 1.195643087927643 \\
    s_{549,5002} &\leq 0.0246578041355137 \\
     i^{(2)}_{549,5002} &\leq 39.69
\end{align*}

\subsection{i2 definition}\label{appendix:i2def}

Let the admittance matrix for line $\{km\}$ be
\begin{equation*}
    Y := \begin{pmatrix}
        \left(y + \frac{y^{sh}}{2}\right)\frac{1}{\tau^{2}} & -y \frac{1}{\tau e^{-j\sigma}} \\
        -y \frac{1}{\tau e^{j\sigma}} & y + \frac{y^{sh}}{2}
        \end{pmatrix} 
        = G + j B
\end{equation*}
where $y = g + jb $ denotes its series admittance, shunt admittance is denoted by $y^{sh} = g^{sh} + j b^{sh}$, and $N := \tau e^{j\sigma}$ denotes the transformer ratio of magnitude $\tau > 0$ and phase shift angle $\sigma$. By Ohm's Law the current flowing from bus $k$ to $m$ is given by $I_{km}
= \frac{y}{\tau}\left( \frac{1}{\tau}V_{k} - e^{j \sigma} V_{m} \right) + V_{k} \frac{y^{sh}}{2\tau^{2}}$. It can be shown (c.f.~\cite{cc}) that $i^{2}_{km} := |I_{km}|^{2}$ can be written as
\begin{align}
    i^{2}_{km} &= \alpha_{km} v_{k}^{(2)} + \beta_{km} v_{m}^{(2)} + \gamma_{km} c_{km} + \zeta_{km} s_{km}, \label{appendix:i2def_eq}
\end{align}
where 
\begin{align*}
    \alpha_{km} :&= \frac{1}{\tau^{4}} \left( (g^{2} + b^{2}) + (gg^{sh} + bb^{sh}) + \frac{(g^{sh2} + b^{sh2})}{4} \right) \\
    \beta_{km} :&= \frac{(g^{2} + b^{2})}{\tau^{2}} \\
    \gamma_{km} :&= \frac{1}{\tau^{3}} \cos(\sigma)( -2 (g^{2} + b^{2}) - (gg^{sh} + bb^{sh}) )  \\
    &+ \frac{1}{\tau^{3}} \sin(\sigma)(bg^{sh} - gb^{sh})\\
    \zeta_{km} :&= \frac{1}{\tau^{3}} \sin(\sigma)( -2 (g^{2} + b^{2}) - (gg^{sh} + bb^{sh}) ) \\
    &- \frac{1}{\tau^{3}}\cos(\sigma)(bg^{sh} - gb^{sh}).
\end{align*}

\subsection{Proof of Proposition~\ref{prop:jabrouter}}\label{appendix:jabrouter}
By \eqref{eq:rotatedrewrite} we have that the Jabr inequality $c_{km}^{2} + s_{km}^{2} \leq v_{k}^{(2)} v_{m}^{(2)}$ can be written as $||(2 c_{km}, 2s_{km}, v_{k}^{(2)} - v_{m}^{(2)} )||_{2} \leq v_{k}^{(2)} + v_{m}^{(2)}$. Hence, taking $\lambda = (1,0,0)^{\top}$, and using \eqref{eq:genericouter} we have
\begin{equation}\label{loss:3}
    v_{k}^{(2)} + v_{m}^{(2)} - 2c_{km} \geq 0.
\end{equation}
On the other hand, since by summing up equations \eqref{SOCP:def_from_activepower} and \eqref{SOCP:def_to_activepower} we have $P_{km} + P_{mk} = G_{kk} v_{k}^{(2)} + G_{mm} v_{m}^{(2)} + 2 G_{km} c_{km}$, which can be lower bounded by $\min \{G_{kk},G_{mm}\} (v_{k}^{(2)} + v_{m}^{(2)} - 2 c_{km})$, the inequality \eqref{loss:3} implies $P_{km} + P_{mk} \geq 0$.

\subsection{Proof of Proposition~\ref{proposition:projection}}\label{appendix:proofprojection}
We solve the following convex QCQP analitically:
\begin{align*}
\min \frac{1}{2} || (x,s) - (\overline{x},\overline{s}) ||^{2} \hspace{2em} \text{s.t.} \hspace{2em} ||x||^{2} \leq s^2.
\end{align*}
For a multiplier $\lambda \in \mathbb{R}_{+}$, consider the Lagrangean $L:= \frac{1}{2} || (x,s) - (\overline{x},\overline{s}) ||^{2} + \lambda (||x||^{2} -s^2)$. The first order conditions are $x-\overline{x} + 2x \lambda = 0$ and $s - \overline{s} - 2 s  \lambda = 0$. Then the second FOC implies $s > 0$ (since we assumed $\overline{s}>0$). Therefore, $\lambda = \frac{(s - \overline{s})}{2s} > 0$, and by substituting $\lambda$ in the first FOC we obtain $x = \frac{s}{2s - \overline{s}} \overline{x}$, where $2s - \overline{s}>0$ holds since $\lambda > 0$. Since the projection must be on the boundary of the closed cone $C$ we have that $(x,s)$ satisfies $||x||^{2} = s^{2} \iff ||\overline{x}|| = 2s - \overline{s}$, i.e., $s_{0} = \frac{||\overline{x}|| + \overline{s}}{2}$. Therefore, $x_{0} = \frac{||\overline{x}|| + \overline{s}}{2} \frac{\overline{x}}{||\overline{x}||}$.

The normal vector of the supporting hyperplane of $C$ at $(x_{0},s_{0})$ is given by
\begin{align*}
(\overline{x},\overline{s}) - (x_{0},s_{0}) &= \left( x' \left( 1 - \frac{s_{0}}{||\overline{x}||} \right), \frac{\overline{s}-||\overline{x}||}{2} \right) \\
&= \left( \frac{||\overline{x}|| - \overline{s}}{2} \right) \left( \frac{\overline{x}}{||\overline{x}||}, -1 \right)
\end{align*}
Finally, by strong duality the separating hyperplane of interest is $a^{t} x + b s \leq 0$, hence $\overline{x}^{t} x  - s || \overline{x} || \leq 0$.

\subsection{Proof of Proposition~\ref{prop:thermalcuts}}\label{appendix:proofthermalcuts}
Since $(x')^{2} + (y')^{2} > r^{2}$, there exists some $0 < t_{0} < 1$ such that $(t_{0} x')^{2} + (t_{0} y')^{2} = r^{2}$. It can be readily checked that $t_{0}(x',y')$ is the projection of $(x',y')$ onto $S_{r}$. Therefore, the normal vector of the separating hyperplane is $(1-t_{0})(x',y')$ and the RHS is $(1-t_{0})t_{0}||(x',y')||_{2}^{2}$, in other words, $(x')x + (y')y \leq r ||(x',y')||_{2}$ is the desired valid inequality since $0<t_{0}<1$.

\tiny Sun.Feb..4.182013.2024

\end{document}